\documentclass[12pt,reqno]{amsart}
\usepackage{amssymb, amsfonts, amsbsy, latexsym, epsfig, color}

\newtheorem{Thm}{Theorem}[section]

\newtheorem{corollary}[Thm]{Corollary}

\newtheorem{definition}[Thm]{Definition}
\newtheorem{remark}[Thm]{Remark}

\newtheorem{example}[Thm]{Example}
\newtheorem{theorem}[Thm]{Theorem}

\newcommand{\bitem}{\begin{itemize}}
\newcommand{\eitem}{\end{itemize}}
\newcommand{\benum}{\begin{enumerate}}
\newcommand{\eenum}{\end{enumerate}}
\newcommand{\beq}{\begin{equation}}
\newcommand{\eeq}{\end{equation}}

\newcommand{\absip}[2]{| \langle#1,#2\rangle |}
\newcommand{\norm}[1]{\|#1\|}

\newcommand{\Id}{\mbox{\rm Id}}

\def\NN{\mathbb{N}}
\def\ZZ{\mathbb{Z}}
\def\RR{\mathbb{R}}
\def\ZZ{\mathbb{Z}}

\def\cR{{\mathcal{R}}}

\def\cH{{\mathcal{H}}}

\def\SSn{\mathbb S}
\def\cHn{\mathcal H}

\def\cH{\mathcal{H}}

\newcommand{\gk}[1]{{\color{black}{#1}}}

\begin{document}

\title{A quantitative notion of redundancy for infinite frames}

\author[J. Cahill]{Jameson Cahill}
\address{Department of Mathematics, University of Missouri, Columbia, MO 65211-4100, USA}
\email{jameson.cahill@gmail.com}

\author[P. G. Casazza]{Peter G. Casazza}
\address{Department of Mathematics, University of Missouri, Columbia, MO 65211-4100, USA}
\email{pete@math.missouri.edu}

\author[A. Heinecke]{Andreas Heinecke}
\address{Department of Mathematics, University of Missouri, Columbia, MO 65211-4100, USA}
\email{ah343@mail.missouri.edu}

\thanks{The authors were supported by NSF DMS 0704216 and 1008183.}

\begin{abstract}
Bodmann, Casazza and Kutyniok introduced a quantitative notion
of redundancy for finite frames - which they called
{\em upper and lower redundancies} - that match better with
an intuitive understanding of redundancy for finite frames \gk{ in a Hilbert
space}.
 The objective of this paper is to see how much of this
 theory generalizes to infinite frames.
  \end{abstract}



\maketitle

\section{Introduction}

The customary notion of redundancy for a finite frame $\{\phi_i\}_{i=1}^N$ in $\cH_n$ is
to use $\frac{N}{n}$.  Many people have felt for a long time that this was not really
satisfactory since it assigns {\em redundancy 2} to each of the following frames (where
$\{e_i\}_{i=1}^n$ is an orthonormal basis for $\cH_n$):

\[ \Phi_{1} = \{e_1,e_1,e_2,e_2,\ldots,e_n,e_n\};\]
\[ \Phi_2 = \{e_1,\ldots,e_1,e_2,e_3,\dots,e_n\},\ \ \mbox{where $e_1$ occurs $(n+1)$-times,}\]
\[ \Phi_3 = \{e_1,0,e_2,0,\ldots,e_n,0\}.\]

The frame $\Phi_1$ has redundancy 2 and is a disjoint union of two spanning sets and
a disjoint union of two linearly independent sets.  This description of redundancy is
informative.  But for $\Phi_2$, the frame is heavily concentrated in one dimension of the
space.  In particular, this frame is made up of just one spanning set and it requires
(n+1)-linearly independent sets to represent it.  Finally, the frame $\Phi_3$ is made
up of one orthonormal basis and a collection of zero vectors.  Assigning this frame
redundancy 2 is quite misleading.  Although zero vectors are important in some
areas of frame theory, such as filter bank theory, counting them in redundancy gives
no useful information.  What is important, is to keep track of the number of zero vectors
while not letting them artificially increase redundancy.
\medskip

In this paper, we generalize the results of \cite{BCK} by applying their quantitative
notion of redundancy for finite frames \gk{in a Hilbert space $\cH$},  to infinite frames.
Most of the results carry over easily, but a few fail in this setting .

Concerning infinite-dimensional Hilbert spaces, 
much work has been done on the idea of {\em deficits, excesses
and redundancy} \cite{BCHL03, BCHL03a,BCHL06,BCHL06a,BCHL06b, BCL09}.
In \cite{BL07}, the authors provide
a meaningful quantitative notion of redundancy which applies to general
infinite frames. In their work, redundancy
is defined as the reciprocal of a so-called frame measure function, which is a
function of certain averages of inner products of frame elements with their
corresponding dual frame elements.  More recently, in \cite{BCL09}, it is shown that
$\ell_1$-localized frames satisfy
several properties intuitively linked to redundancy such as that any frame with
redundancy \gk{greater} than one should contain in it a frame with redundancy arbitrarily
close to one, the redundancy of any frame for the whole space should be greater than
or equal to one, and that the redundancy of a Riesz basis should be exactly one,
were proven for this notion.  Our approach is slightly different in that we are interested  
in how many spanning sets or linearly independent sets are in the frame.  However,
our notion does not capture much information about infinite frames whose frame
vectors are not bounded.  We will give examples to show the problems with 
our notion of redundancy for unbounded frames.

\subsection{Review of Frames}

We start by fixing our terminology while briefly reviewing the basic definitions
related to frames. \gk{Let} $\cHn$ denote an $n$-dimensional
real or complex Hilbert space and $\cH$ denotes
a finite or infinite dimensional Hilbert space.  A {\em frame} for a Hilbert
space $\cH$ is a family of vectors $\{\phi_i\}_{i\in I}$
(with $|I|$ finite or infinite)  for which there exists constants $0 < A \le B < \infty$
such that
\[
A\norm{x}^2 \leq \sum_{i\in I} |\langle x, \varphi_i \rangle |^2 \leq B\norm{x}^2
\quad \mbox{for all } x \in \cHn.
\]
When $A$ is chosen as the largest possible value and $B$ as the smallest
for these inequalities to hold,
then we call them the {\em (optimal) frame constants}.
If $A$ and $B$ can be chosen as $A=B$, then the frame is called {\em $A$-tight}, and if
$A=B=1$ is possible, $\Phi$ is a {\em Parseval frame}. A frame is called
{\em equal-norm}, if there exists some $c>0$ such that $\|\varphi_i\|=c$ for all
$i=1,\ldots,N$, and it is {\em unit-norm} if $c=1$.

Apart from providing redundant expansions, frames can also serve as an analysis tool.
In fact, they allow the analysis of data by studying the associated {\em frame coefficients}
$(\langle x, \varphi_i \rangle)_{i\in I}$, where the operator $T_\Phi$
defined by
$T_\Phi: \cHn \to \ell_2I)$, $x \mapsto (\langle x,\varphi_i\rangle)_{i\in I}$
is called the \emph{analysis operator}. The adjoint $T^*_\Phi$ of the analysis operator is typically
referred to as the {\em synthesis operator} and satisfies $T^*_\Phi((c_i)_{i\in I}) = \sum_{i\in I}
c_i\varphi_i$.
The main operator associated with a frame, which provides a stable reconstruction process, is the
{\em frame operator}
\[
S_\Phi=T^*_\Phi T_\Phi : \cHn \to \cHn, \quad x \mapsto \sum_{i\in I} \langle x,\varphi_i\rangle \varphi_i,
\]
a positive, self-adjoint invertible operator on $\cHn$. In the case of a Parseval frame, we have $S_\Phi=\Id_{\cHn}$.
In general, $S_\Phi$ allows reconstruction of a signal $x \in \cHn$ through the
reconstruction formula
\beq \label{eq:expansion}
x = \sum_{i\in I}\langle x,S_\Phi^{-1} \varphi_i\rangle \varphi_i.
\eeq
The sequence $(S_\Phi^{-1} \varphi_i)_{i=1}^N$ which can be shown to form a frame itself, is often
referred to as the {\em canonical dual frame}.

We note that the choice of coefficients in the expansion \eqref{eq:expansion}
is generally not the only possible one. If the frame is linearly dependent --
which is typical in applications --
then there exist infinitely many choices of coefficients $(c_i)_{i=1}^N$ leading to expansions of
$x \in \cH$ by
\beq \label{eq:sparseexpansion}
x = \sum_{i\in I} c_i \varphi_i.
\eeq
This fact, for instance, ensures resilience to erasures and noise. The particular choice of coefficients
displayed in \eqref{eq:expansion} is the smallest in $\ell_2$ norm \cite{Chr03}, hence contains the least energy.
A different paradigm has recently received rapidly increasing attention \cite{BDE09}, namely to choose
the coefficient sequence to be sparse in the sense of having only few non-zero entries, thereby allowing
data compression while preserving perfect recoverability.


For a more extensive introduction to frame theory, we refer the interested reader to the books
\cite{Dau92,Mal98,Chr03} as well as to the survey papers \cite{KC07a,KC07b}.

\subsection{\gk{An Intuition-Driven Approach to Redundancy}}
\label{subsec:intuition}

 In order to properly define a quantitative notion
of redundancy, the authors of \cite{BCK} first gave a list of desiderata that a notion is required
to satisfy.

\subsection{Desiderata}
\label{subsec:desiderata}

Summarizing and analyzing the requirements we have discussed, we state the following list of
desired properties for an upper redundancy  $\cR^+_\Phi$ and a lower redundancy $\cR^-_\Phi$
of a frame $\Phi = (\varphi_i)_{i\in I}$ for a finite or infinite dimensional
real or complex Hilbert space $\cH$.

\renewcommand{\labelenumi}{{\rm [D\arabic{enumi}]}}

\begin{enumerate}\item\label{D0}{\em Zero Vectors.}  Redundancy should not count
zero vectors.
\item\label{DEE} {\em Generalization.} If $\Phi$ is an equal-norm Parseval frame, then
in this special case the customary notion of redundancy shall be attained, i.e., $\cR^-_\Phi = \cR^+_\Phi $.
\item\label{DNyquist} {\em Nyquist Property.} The condition $\cR^-_\Phi = \cR^+_\Phi$ shall characterize
tightness of a normalized version of $\Phi$, thereby supporting the intuition that upper and lower
redundancy being different implies `non-uniformity' of the frame. In particular, $\cR^-_\Phi =
\cR^+_\Phi = 1$ shall be equivalent to orthogonality as the `limit-case'.
\item\label{Duplow} {\em Upper and Lower Redundancy.} Upper and lower redundancy shall be
`naturally' related by $0 < \cR^-_\Phi \le \cR^+_\Phi < \infty$.
\item\label{Dadditiv} {\em Additivity.} Upper and lower redundancy shall be subadditive and superadditive,
respectively, with respect to unions of frames. They shall be additive provided that the redundancy
is uniform, i.e., $\cR^-_\Phi=\cR^+_\Phi$.
\item\label{DInvariance} {\em Invariance.} Redundancy shall be invariant under the action of a unitary
operator on the frame vectors, under scaling of single frame vectors, as well as under permutation,
since intuitively all these actions should have no effect on, for instance, robustness against erasures,
\gk{which is one property redundancy shall intuitively measure.}
\item\label{Daverob} {\em Spanning Sets.} The lower redundancy shall measure the maximal number of
spanning sets of which the frame consists. This immediately implies that the lower
redundancy is a measure for robustness of the frame against erasures in the sense that any set of
a particular number of vectors can be deleted yet leave a frame.
\item\label{Dmaxrob} {\em Linearly Independent Sets.} The upper redundancy shall measure the minimal
number of linearly independent sets of which the frame consists.
\end{enumerate}

It is straightforward to verify that for the special type of frames consisting of orthonormal basis vectors, each
repeated a certain number of times, the upper and lower redundancies given by the maximal or minimal
number of repetitions satisfy these conditions. The challenge is now to extend this definition to all
frames in such a way that as many of these properties as possible are preserved.

\renewcommand{\labelenumi}{{\rm (\roman{enumi})}}


\section{Defining Redundancy}

\subsection{Definitions}

As explained before, we first introduce a local redundancy given in \cite{BCK}, which
encodes the concentration of frame vectors around one point. Since the
norms of the frame vectors do not matter for concentration, we normalize
the given frame and also consider only points on the unit sphere $\SSn=\{x \in \cHn: \|x\|=1\}$
in $\cH$. Hence another way to view local redundancy is by considering
it as some sort of density function on the sphere.  A consequence of normalizing the
frame vectors, is that our new set may no longer be Bessel.

We now define a notion of local redundancy. \gk{For this, we remark that} throughout the paper, we
let \gk{$\langle y \rangle$ denote the span of some $y \in \cHn$ and $P_{\langle y \rangle}$
the orthogonal projection onto $\langle y \rangle$.}

\begin{definition}
Let $\Phi = (\varphi_i)_{i\in I}$ be a frame for $\cH$. For
each $x \in \SSn$, the {\em redundancy function}
$\cR_\Phi : \SSn \to \RR^+$ is defined by
\[
\cR_\Phi(x) = \sum_{i\in I} \|P_{\langle \varphi_i \rangle} (x)\|^2
.
\]
\end{definition}

We might think about the function $\mathcal R_\Phi$ as a redundancy pattern on the sphere,
which measures redundancy at each single point. Also notice that this notion is reminiscent
of the fusion frame condition \cite{CKL08}, here for rank-one projections.

In contrast to \cite{BCK}, this redundancy function may not assume its maximum or
minimum on the unit sphere and in general both the max and min of this function
could be infinite.

We now define:

\begin{definition} \label{def:upplowred}
Let $\Phi = (\varphi_i)_{i\in I}$ be a frame for  $\cH$.
Then the {\em upper redundancy of $\Phi$} is defined by
\[
\cR^+_\Phi = \sup_{x \in \SSn} \cR_\Phi(x)
\]
and the {\em lower redundancy of $\Phi$} by
\[
\cR^-_\Phi = \inf_{x \in \SSn} \cR_\Phi(x).
\]
Moreover, $\Phi$ has a {\em uniform redundancy}, if
\[
\cR^-_\Phi = \cR^+_\Phi.
\]
\end{definition}

This notion of redundancy hence equals the upper and lower frame bound of the
normalized version of the frame - which could now be infinite. 


\section{The Case of Infinite Redundancy}
 
\subsection{Main Result}

With the previously defined quantitative notion of upper and lower redundancy, we can
now verify the  properties from Subsection \ref{subsec:desiderata} which hold (and those
which do not hold) in the infinite dimensional setting.
\begin{theorem}
\label{T1}
Let $\Phi = (\varphi_i)_{i\in I}$ be a frame for an $\infty$-dimensional real or complex Hilbert space   $\cHn$ and assume that $\cR^+< \infty$.
\begin{enumerate}
\item[{\rm [D1]}] {\em Generalization.} If $\Phi$ is an equal-norm Parseval frame, then
\[
\cR^-_\Phi = \cR^+_\Phi.
\]
\item[{\rm [D2]}] {\em Nyquist Property.} The following conditions are equivalent:
\bitem
\item[{\rm (i)}] We have $\cR^-_\Phi = \cR^+_\Phi$.
\item[{\rm (ii)}] The normalized version of $\Phi$ is tight.
\eitem
Also the following conditions are equivalent.
\bitem
\item[{\rm (i')}] We have $\cR^-_\Phi = \cR^+_\Phi = 1$.
\item[{\rm (ii')}] $\Phi$ is orthogonal.
\eitem
\item[{\rm [D3]}] {\em Upper and Lower Redundancy.} We have
\[
0 < \cR^-_\Phi \le \cR^+_\Phi < \infty.
\]
\item[{\rm [D4]}] {\em Additivity.} For each orthonormal basis $(e_i)_{i=1}^n$,
\[
\cR^\pm_{\Phi\cup(e_i)_{i=1}^n} = \cR^\pm_\Phi + 1.
\]
Moreover, for each frame $\Phi'$ in $\cHn$,
\[
\cR^-_{\Phi\cup\Phi'} \ge \cR^-_\Phi + \cR^-_{\Phi'}
\quad \mbox{and} \quad
\cR^+_{\Phi\cup\Phi'} \le \cR^+_\Phi + \cR^+_{\Phi'}.
\]
In particular, if $\Phi$ and $\Phi'$ have uniform redundancy, then
\[
\cR^-_{\Phi\cup\Phi'} = \cR_\Phi + \cR_{\Phi'} = \cR^+_{\Phi\cup\Phi'}.
\]
\item[{\rm [D5]}] {\em Invariance.}
Redundancy is invariant under application of a unitary operator $U$ on $\cHn$, i.e.,
\[
\cR^\pm_{U(\Phi)} = \cR^\pm_{\Phi},
\]
under scaling of the frame vectors, i.e.,
\[
\cR^\pm_{(c_i \varphi_i)_{i=1}^N} = \cR^\pm_{\Phi}, \quad c_i \mbox{ scalars},
\]
and under permutations, i.e.,
\[
\cR^\pm_{(\varphi_{\pi(i)})_{i=1}^N} = \cR^\pm_{\Phi}, \quad \pi \in S_{\{1,\ldots,N\}},
\]
\item[{\rm [D6]}] {\em Spanning Sets.} In the finite setting, $\Phi$ contains
$\lfloor \cR^-_\Phi\rfloor$ disjoint spanning sets.
In the infinite dimensional setting, this property fails as we will show with
an example.
\item[{\rm [D7]}] {\em Linearly Independent Sets.} 
$\Phi$ 
can be partitioned into $\lceil \cR^+_\Phi\rceil$ linearly independent sets.
\end{enumerate}
\end{theorem}

\begin{proof}
The properties [D1] is true because redundancy is
the upper and lower frame bounds of the normalized version
of the frame.

The first part of [D2] is true by definition and for the
second part of [D3], it is well known that a unit norm Parseval frame must
be an orthonormal basis.

Property [D4] follows easily from the argument in \cite{BCK}.

Property [D5] is ovbious.

Property [D6] fails as we will see in the next section.

[D7] follows from Theorem 4.2 of \cite{CKLV} which states that:  Every Bessel sequence
$\{\varphi_i\}_{i\in I}$ with Bessel bound $B$ and $\|\varphi_i\|\ge c$ for all $i\in I$ 
(in our case $c=1$), can be decomposed into $\lceil B/c^2\rceil$ linearly independent
sets.
\end{proof}

The redundancy function gives little information near the extreme
cases - as was true in the finite dimensional case, as the following example shows.

\begin{example}
\label{example:F3}
{\rm We add a example in which the frame is not merely composed of vectors from the unit basis
$\{e_1, \ldots, e_n\}$. Letting $0 < \varepsilon < 1$, we choose $\Phi_4 = (\varphi_i)_{i\in I}$
as
\[
\varphi_i = \left\{\begin{array}{rcl}
e_1 & : & i=1,\\
\sqrt{1-\varepsilon^2} e_1 + \varepsilon e_i & : & i \neq 1\\
e_i&:& i>N.
\end{array} \right.
\]
This frame is strongly concentrated around the vector $e_1$. We first observe that
\[
\cR_{\Phi_3}(e_1) = \sum_{i=1}^N \|P_{\langle \varphi_i \rangle}(e_1)\|^2
= 1 + \sum_{i=2}^N \absip{e_1}{\sqrt{1-\varepsilon^2} e_1 + \varepsilon e_i}^2
= 1 + (N-1)(1-\varepsilon^2).
\]
However, this is not the maximum, which is in fact attained at the average point of the
frame vectors. But in order to avoid clouding the intuition by technical details,
we omit this analysis, and observe that
\[
1 + (N-1)(1-\varepsilon^2) \le \cR_{\Phi_4}^+ < N.
\]
Since
\[
\cR_{\Phi_4}(e_2) = \sum_{i=1}^N \|P_{\langle \varphi_i \rangle}(e_2)\|^2
= \sum_{i=2}^N \absip{e_2}{\sqrt{1-\varepsilon^2} e_1 + \varepsilon e_i}^2
= \varepsilon^2,
\]
we can conclude similarly, that
\[
0 < \cR_{\Phi_4}^- \le \varepsilon^2.
\]}
\end{example}
The frame $\Phi_3$ shows that the new redundancy notion gives little information
near the extreme cases:  $\cR^- \approx 0$ and $\cR^+$ large,
but becomes increasingly more accurate as $\cR^-$ and $\cR^+$ become closer to one another.
By [D2], the frame $\Phi_4$ is not orthogonal, nor is it tight. [D6] is not applicable
for this frame, since $\lfloor \cR_{\Phi_4}^- \rfloor = 0$ although there does exist a
partition into one spanning set.  Now, [D7] implies that this frame can be partitioned into $N-1$ linearly
independent sets. Again, we see that we can do better than this by merely taking the
whole frame which happens to be linearly independent. As before, we observe that [D7]
is not sharp for large values of $\cR^+$.  However, these become increasingly accurate
as $\cR^-$ and $\cR^+$ approach each other.



\section{Infinite Equal Norm Parseval Frames}

It is easy to construct infinite equal norm
Parseval frames for which the norms of the vectors are aribitrarily
close to one.
\vskip12pt

\begin{theorem}
For any $r\le 1$, there is an equal norm Parseval frame $\{\varphi_i\}_{i=1}^{\infty}$
for $\ell_2$ with $\|\varphi_i\|^2 = r$, for all $i=1,2,\ldots$.
\end{theorem}

\begin{proof}
Given the orthonormal basis $\{e^{2\pi int}\}_{n\in \ZZ}$, let $E\subset [0,1]$ be
a measurable set for which $|E|=r$.  Then
\[ \left \{ e^{2\pi int}\chi_E \right \}_{n\in \ZZ}\]
is a Parseval frame of norm $r$ vectors.
\end{proof}


\section{Infinite Parseval Frames}

Since {\em linear independence} is so weak in the infinite dimensional setting,
we will now see that (equal norm Parseval) frames can have some surprising
properties.  This will affect our work in this area.

\begin{example}\label{E1}
For every natural number $k\in \NN$
there is an equal norm Parseval frame for $\ell_2$ which can be written as
$j$-linearly independent and disjoint spanning sets, for all $j=1,2,\ldots,k$.
\end{example}

\begin{proof}
It is straightforward to choose families of vectors $\{f_{ij}\}_{i,j=1}^{\infty}$ satisfying:
\vskip12pt
(1)  The vectors $\{f_{ij}\}_{i,j=1}^{\infty}$ are linearly independent. 
\vskip12pt
(2)   For each $j=1,2,\ldots$, we have that span $\{f_{ij}\}_{i=1}^{\infty}$ is dense in $\ell_2$.  
\vskip12pt

It follows that if we apply Grahm-Schmidt to $\{f_{ij}\}_{i=1}^{\infty}$ 
for each $j=1,2,\ldots$, we get a sequence of orthonormal basis 
$\{g_{ij}\}_{i=1}^{\infty}$ for $\ell_2$, with the property that $\{g_{ij}\}_{i,j=1}^{\infty}$
is a linearly independent set.
Fix $k\in \NN$ and consider the family:
\[ \left \{ \frac{1}{\sqrt{k}}f_{ij}\right \}_{i=1,j=1}^{\  \infty,\ k}.\]
This family clearly has the desired properties.
\end{proof}

\begin{example}
There is a Parseval frame for $\ell_2$ which can be written as
$j$-linearly independent and disjoint spanning sets, for all $j=1,2,\ldots,\infty$.
(Note that $j=\infty$ is included here).
\end{example}

\begin{proof}
We use the family $\{g_{ij}\}_{i,j=1}^{\infty}$ from Example \ref{E1} and form
the Parseval frame
\[ \left \{ \frac{1}{2^j}g_{ij} \right \}_{i,j=1}^{\infty}.\]

This is the required family.
\end{proof}

\section{More on the Infinite Version of Property [D6]}

In this section we will further examine property [D6].  Unfortunately, it is 
dangerously
close to Kadison-Singer.  First, we give an alternative proof of Corollary 2.4
of \cite{BCPS09}.  

\begin{definition}
A family of vectors $\{\varphi_i\}_{i=1}^{\infty}$ is $\omega$-{\bf independent}
if whenever 
\[ \sum_{i=1}^{\infty}a_i\varphi_i =0,\]
it follows that $a_i=0$, for all $i=1,2,\ldots$.  If we have this property only for
all $\{a_i\}_{i=1}^{\infty}\in \ell_2$, we say the family of vectors is $\ell_2$-{\bf independent}.
\end{definition}

\begin{theorem}
Let $\{Pe_i\}_{i=1}^{\infty}$ be a Parseval frame in $\cH$.  If $I\subset \NN$, the following
are equivalent:

1.    The family $\{Pe_i\}_{i\in I}$ spans $P(\cH)$.

2.  The family $\{(I-P)e_i\}_{i\in I^c}$ is $\ell_2$-independent. 
\end{theorem}

\begin{proof}
\noindent $(1)\Rightarrow (2)$:  
Assume that $\{(I-P)e_i\}_{i\in I^c}$ is not $\ell_2$-independent.  Then there exists
scalars $\{b_i\}_{i\in I^c}\in \ell_2$ so that
\[ \sum_{i\in I^c}b_i(I-P)e_i =0.\]
It follows that
\[ f = \sum_{i\in I^c}b_ie_i = \sum_{i\in I^c}b_iPe_i \in P(\cH).\]
Thus,
\[ \langle f,Pe_j\rangle = \langle Pf,e_j\rangle = \sum_{i\in I^c}b_i\langle e_i,e_j\rangle
=0,\ \ \mbox{if $j\not= i$. i.e. if $j\in I$}.\]
So $f \perp span\ \{Pe_i\}_{i\in I}$, and this family is not spanning for $P(\cH)$.

\vskip12pt
\noindent $(2)\Rightarrow (1)$:
First assume there is an $f\in P(\cH)$ so that $f \perp span\ \{Pe_i\}_{i\in I}$.
Then, $f=\sum_{i\in I}a_iPe_i$.  Also,
\[ \langle f,Pe_i\rangle = \langle Pf,e_i\rangle = \langle f,e_i\rangle =0,
\ \ \mbox{for all $i\in I$}.\]
Hence, $f=\sum_{i\in I^c}b_ie_i$, with not all $b_i=0$ and $\{b_i\}_{i\in I^c}\in \ell_2$.  Thus,
\[ \sum_{i\in I^c}b_ie_i = f = Pf = \sum_{i\in I^c}b_iPe_i.\]
i.e.
\[ \sum_{i\in I^c}b_i(I-P)e_i =0.\]
That is, $(I-P)e_i\}_{i\in I^c}$ is not $\ell_2$-independent.
\end{proof}

\begin{corollary}
The property [D6] is true for unit norm 2-tight frames
 if and only if whenever $\{\varphi_i\}_{i=1}^{\infty}$ is a
unit norm 2-tight frame then there is a partition $\{I_1,I_2\}$ of $\NN$ so that
$\{\varphi_i\}_{i\in I_j}$ is $\ell_2$-independent for $j=1,2$.
\end{corollary}

If we examine the proof above, we see that what is proved is really
the following:

\begin{corollary}
Let $\{e_i\}_{i\in I}$ be an orthonormal basis for
$\cH$ and $\{Pe_i\}_{i\in I}$ be a Parseval frame for $P(\cH)$.
If  
\[ \varphi \in span\ \{e_i\}_{i\in J}\cap P(\cH),\]
then $\varphi \perp span\ \{Pe_i\}_{i\in J^c}$.

In particular,
the following are equivalent for a subset $J\subset I$:

(1) We have
\[ span\ \{e_i\}_{i\in J}\cap P(\cH) \not= \{0\},\]

(2)  We have
\[ span\ \{Pe_i\}_{i\in J^c} \not= \cH.\]

(3)  The family $\{(I-P)e_i\}|_{i\in J}$ is not $\ell_2$-independent.
\end{corollary}

\begin{remark}
Note that the above corollary unifies the finite linearly
independent result with the infinite one.  i.e.  The same theorem above
holds with $|I|$ finite and {\em linearly independent} for part (3).
\end{remark}

\begin{remark}
The above also raises the question if there is an infinite dimensional
Rado-Horn Theorem.  But we are not sure what it should say at this time.
\end{remark}

We note that property [D6] is an infinite dimensional version of a result
from \cite{CFMT}.  In this paper, using variations of the discrete
Fourier transform matrices, the authors construct families of unit norm
2-tight frames for $\cH_n$, so that whenever you partition the frame vectors into two subsets,
the lower Riesz bound of one of the subsets is on the order of $1/n$.  
For a counter-example to  [D6], we are
looking for unit norm 2-tight frames for $\ell_2$ so that whenever you divide the
frame vectors into two sets, one of them is not $\ell_2$-independent.

Finally, let us observe that we can find an equal norm Parseval frame with
the above properties.  This example is due to Bodmann, Casazza, Paulsen,
and Speegle.

\begin{example}
For any $E\subset [0,1]$ measurable, the family
\[ \left \{ e^{2\pi int}\chi_E \right \}_{n\in \ZZ},\]
can be written as $k$, $\ell_2$-independent spanning sets for all
$k=1,2,\ldots,\infty$.
\end{example}

\section{Some Notes}

It is possible that there is a better notion of redundancy than that given in
\cite{BCK}.  The problem with that notion is that if we apply 
an invertible operator
to a frame, we get different redundancy.  Intuitively, 
this should not give a different
value.  A possible alternative definition is:

\begin{definition}
Given a frame $\Phi=\{\varphi_i\}_{i=1}^N$ in $\cH_n$ with frame operator $S$, 
let
or
each $x \in \SSn$, the {\em redundancy function}
$\cR_\Phi : \SSn \to \RR^+$ is defined by
\[
\cR_\Phi(x) = \sum_{i=1}^N \|P_{\langle S^{-1/2}(\varphi_i) \rangle} (x)\|^2
.
\]
\end{definition}

\begin{definition} \label{def:upplowred}
Let $\Phi = (\varphi_i)_{i\in I}$ be a frame for  $\cH$.
Then the {\em upper redundancy of $\Phi$} is defined by
\[
\cR^+_\Phi = \sup_{x \in \SSn} \cR_\Phi(x)
\]
and the {\em lower redundancy of $\Phi$} by
\[
\cR^-_\Phi = \inf_{x \in \SSn} \cR_\Phi(x).
\]
Moreover, $\Phi$ has a {\em uniform redundancy}, if
\[
\cR^-_\Phi = \cR^+_\Phi.
\]
\end{definition}

This notion of redundancy equals the upper and lower frame bounds of the 
normalized version of the canonical Parseval frame to $\Phi$.

This definition seems to lose some of the properties of the original definition - which
needs to be checked - such as

[D4]  Do these hold?  Especially, if we add a Parseval frame to a frame, 
do these
redundancies increase by 1?

[D5]  Does \[
\cR^\pm_{(c_i \varphi_i)_{i=1}^N} = \cR^\pm_{\Phi}, \quad c_i 
\mbox{ scalars},
\]

Everything else seems to hold.  But we do pick up a new result that 
redundancy is
invariant under application of an invertible operator.

\begin{theorem}
If $\Phi = \{\varphi_i\}_{i=1}^N$ is a frame for $\cH_n$ and $T$ is an 
invertible
operator on $\cH_n$, then for all $x\in \cH_n$ we have
\[ \cR_{T(\Phi)}^+(x)= \cR_{\Phi}^+(x),\  \ \mbox{and}\ \  
\cR_{T(\Phi)}^-(x) = 
\cR_{\Phi}^-(x).\]
\end{theorem}

\begin{proof}
Let $S$ (resp. $S_T$) be the frame operator for $\Phi$ (resp. $T(\Phi)$).  
Then
$S^{-1/2}\Phi$ is equivalent to $\Phi$ which is equivalent to $T(\Phi)$ 
which is
equivalent to $S_1^{-1/2}(T\Phi)$.  Since both of these frames are 
Parseval, it follows
from \cite{CKov} that there is a unitary operator $U$ satisfying:
\[ U[(S^{-1/2}(\Phi)] = S_1^{-1/2}(T\Phi).\]
The result is obvious from here.
\end{proof}

\noindent {\bf Alert}:  Unfortunately, this new idea for a definition
does not work.  We will now give an example to show that the upper
frame bound of the normalized version of a Parseval frame is not
related to the number of linearly independent sets we can partition
the family into.

\vskip12pt

\begin{example}  Fix $N$ and let $\{e_i\}_{i=1}^N$ be an orthonormal
basis for $\cH_N$.  We will build a Parseval frame for $\cH_N$ in
pieces.  First, we build the sets
\[ \frac{1}{\sqrt{2N}}e_1+\frac{1}{\sqrt{2N}}e_i,\ \ 
\frac{1}{\sqrt{2N}}e_1-\frac{1}{\sqrt{2N}}e_i\ \ \mbox{for all
$i=2,3,\ldots,N$}.\]
This family is a frame for $\cH_N$ with frame operator having
eigenvectors $\{e_i\}_{i=1}^N$ and respective eigenvalues
\[ \{\frac{N-1}{N},\frac{1}{N},\frac{1}{N} ,\cdots, \frac{1}{N}\}\]
Hence, if we add to this family the vectors
\[ \{\frac{1}{\sqrt{N}}e_1\} \cup \left \{ \sqrt{1-\frac{1}{N}}e_i \right \}_{i=2}^N,\]
then we will have a Parseval frame which can clearly be devided into
3 linearly independent sets.  Namely, divide the first set into
\[ \left \{\frac{1}{\sqrt{2N}}e_1+\frac{1}{\sqrt{2N}}e_i \right \}_{i=2}^N\]
and
\[ \left \{\frac{1}{\sqrt{2N}}e_1-\frac{1}{\sqrt{2N}}e_i \right \}_{i=2}^N\]
and the third set is already linearly independent.  However, if we
normalize the vectors, we get a frame:
\[ \left \{\frac{1}{\sqrt{2}}e_1 \pm \frac{1}{\sqrt{2}}e_i\right \}_{i=2}^N,\]
plus $\{e_i\}_{i=2}^N$.  For this family, if we check the
frame bound at say $e_1$ we get:
\[ \sum_{i=2}^N |\langle e_1,\frac{1}{\sqrt{2}}e_1\pm \frac{1}{\sqrt{2}}
e_i \rangle |^2 = \frac{N-1}{2}.\]
That is, the upper frame bound of the normalized version of this
Parseval frame is unrelated to the number of linearly independent
sets we can divide it into.
\end{example}

\section{Concluding Remarks}

In the case of infinite redundancy, it is possible that our upper frame bound
is infinity.  It is not clear at this time if anything can be concluded from this
case.



\begin{thebibliography}{99}

\bibitem{BCHL03a}  R. Balan, P.G. Casazza, C. Heil and Z. Landau,
{\em Deficits and excesses of frames},
Advances in Computational Mathematics {\bf 18} No. 2-4 (2003) 93--116.

\bibitem{BCHL03}  R. Balan, P.G. Casazza, C. Heil and Z. Landau,
{\em Excesses of Gabor Frames},
Applied and Computational Harmonic Analysis {\bf 14} (2003) 87--106.


\bibitem{BCHL06}  R. Balan, P.G. Casazza, C. Heil and Z. Landau,
{\em Density, overcompleteness and localization of frames:  1.  Theory},
J. Fourier Analysis and Applications {\bf 12} (2006) 105--143.

\bibitem{BCHL06a}  R. Balan, P.G. Casazza, C. Heil and Z. Landau,
{\em Density, overcompleteness and localization of frames:  2.
Gabor Systems},
J. Fourier Analysis and Applications {\bf 12} (2006) 309--344.

\bibitem{BCHL06b}  R. Balan, P.G. Casazza, C. Heil and Z. Landau,
{\em Density, overcompleteness and localization of frames},
Electronic Research Announcements AMS {\bf 12} (2006) 71--86.



\bibitem{BCL09} R. Balan, P. G. Casazza, and Z. Landau, {\em Redundancy for localized frames},
preprint.

\bibitem{BL07} R. Balan and Z. Landau, {\em Measure functions for frames}, J. Funct. Anal.
{\bf 252} (2007), 630--676.


\bibitem{BBCE09} R. Balan, B. G. Bodmann,
P. G. Casazza and D. Edidin, {\em Painless Reconstruction from Magnitudes of Frame Coefficients}, J. Fourier Analysis and  Applications {\bf 15} (2009),
488--501.

\bibitem{BCK}  B.G. Bodmann, P.G. Casazza and G. Kutyniok, {\em A quantative
notion of redundancy for finite frames}, preprint.

\bibitem{BCPS09}  B. G. Bodmann, P. G. Casazza, V. Paulsen, and D. Speegle, {\em Spanning
properties of frames}, preprint.

\bibitem{BDE09}
A. M. Bruckstein, D. L. Donoho, and M. Elad,
{\em From Sparse Solutions of Systems of Equations to Sparse Modeling of Signals and Images,}
SIAM Review {\bf 51} (2009), 34--81.

\bibitem{CKLV}  P.G. Casazza, O. Christensen, A.M. Lindner and R. Vershynin,
{\em Frames and the Feichtinger Conjecture}, 


\bibitem{CKL08}
P. G. Casazza, G. Kutyniok, and S. Li,
{\em Fusion Frames and Distributed Processing},
Appl. Comput. Harmon. Anal. {\bf 25} (2008), 114--132.

\bibitem{CKS06}
P. G. Casazza, G. Kutyniok, and D. Speegle, {\em A redundant version of the Rado-Horn Theorem},
Linear Algebra Appl. {\bf 418} (2006), 1--10.


\bibitem{CCHKP09} R. Calderbank, P. G. Casazza, A. Heinecke, G. Kutyniok, and A. Pezeshki,
{\em Sparse Fusion Frames: Existence and Construction}, preprint.

\bibitem{CKov}    P.G. Casazza and J. Kovacevic, {\it Uniform tight frames with erasures}, Advances in Computational Mathematics {\bf Vol.
18, Nos. 2-4} (2003) pp. 387-430.


\bibitem{CL09}  P. G. Casazza and M. Leon, {\em Existence and construction of finite frames
with a given frame operator}, preprint.

\bibitem{CT06}  P.G. Casazza and J.C. Tremain, 

\bibitem{CT09}  P. G. Casazza and J. C. Tremain, {\em A brief introduction to Hilbert-space frame theory and its applications},
preprint posted on www.framerc.org.

\bibitem{Chr03}
O.~Christensen,
\emph{An Introduction to Frames and Riesz Bases},
Birkh\"auser, Boston, 2003.

\bibitem{Dau92}
I.~Daubechies,
\emph{Ten Lectures on Wavelets}, SIAM, Philadelphia, 1992.


\bibitem{DFKLOW}
K. Dykema, D. Freeman, K. Kornelson, D. Larson, M. Ordower, and E. Weber, {\em Ellipsoidal tight
frames and projection decompositions of operators}, Illinois J. Math. {\bf 48}  (2004), 477--489.

\bibitem{Gle57}
A. M. Gleason, {\em Measures on the closed subspaces of a Hilbert space}, J. Math. Mech. {\bf 6} (1957), 885--893.

\bibitem{Hei07}
C. Heil, {\em History and evolution of the Density Theorem for Gabor frames},
J. Fourier Anal. Appl. {\bf 13} (2007), 113--166.


\bibitem{JN35}
P. Jordan and J. von Neumann, {\em On inner products in linear metric spaces},
Annals of Math. {\bf 36} (1935), 719--723.

\bibitem{KC07a}
J. Kova{\v c}evi{\'c} and A. Chebira,
{\em Life beyond bases: The advent of frames (Part I)},
IEEE SP Mag. {\bf 24}  (2007), 86--104.

\bibitem{KC07b}
J. Kova{\v c}evi{\'c} and A. Chebira,
{\em Life beyond bases: The advent of frames (Part II)},
IEEE SP Mag. {\bf 24}  (2007), 115--125.

\bibitem{KC08}  J. Kova{\v c}evi{\'c} and A. Chebira,
{\em An Introduction to Frames}, Foundations and Trends in Signal Processing,
{\bf 2}, No. 1 (2008) 1--94.

\bibitem{Mal98}
S. Mallat,
\emph{A wavelet tour of signal processing},
Academic Press, Inc., San Diego, CA, 1998.

\bibitem{Wei80}
J. Weidmann, {\em Linear Operators in Hilbert Spaces}, Springer-Verlag, Berlin/New York, 1980.

\end{thebibliography}
\end{document}